\documentclass{article}
\usepackage{amsmath,mathrsfs,amssymb,amsthm,amscd}
\usepackage{color,epsfig,latexsym,graphicx,eucal,layout,fancyhdr}
\usepackage[]{fontenc}
\usepackage[all]{xy}
\usepackage{hyperref}
\usepackage{graphics}
\usepackage{a4wide}
\usepackage{palatino}
\usepackage{setspace}

\newcounter{EQNR}
\renewcommand{\contentsname}{Contents\\
{\footnotesize\normalfont(A table of contents should normally not
be included)}}

 \theoremstyle{plain}
 \newtheorem{thm}{Theorem}[section]
 
 \numberwithin{equation}{section} %% Comment out for sequentially-numbered
 \theoremstyle{plain}
 \theoremstyle{plain}
 \theoremstyle{definition}
 \newtheorem{defn}[thm]{Definition}
 
 \theoremstyle{plain}
 \newtheorem{prop}[thm]{Proposition}
 
 \newtheorem{lem}[thm]{Lemma}
 \newtheorem{cor}[thm]{Corollary}
 \newtheorem*{cor*}{Corollary}
 \newtheorem*{conj*}{Conjecture}
 \newtheorem*{thm*}{Theorem}

\newtheorem{example}{Example}

\newcommand{\bl}{\begin{lem}}
\newcommand{\el}{\end{lem}}

\newcommand{\bml}{\begin{multline}}
\newcommand{\eml}{\end{multline}}

\newcommand{\beq}{\begin{equation}}
\newcommand{\eeq}{\end{equation}}
\newcommand{\bp}{\begin{prop}}
\newcommand{\ep}{\end{prop}}
\newcommand{\bd}{\begin{defn}}
\newcommand{\ed}{\end{defn}}
\newcommand{\pf}{\begin{proof}}
\newcommand{\epf}{\end{proof}}

\newcommand{\field}[1]{\ensuremath{\mathbb{#1}}}

\newcommand{\CC}{\field{C}}

\newcommand{\NN}{\field{N}}

\newcommand{\HH}{\field{H}}

\newcommand{\RR}{\field{R}}

\newcommand{\ZZ}{\field{Z}}

\let\Im\relax
\DeclareMathOperator{\Im}{Im}
\let\Re\relax
\DeclareMathOperator{\Re}{Re}

\DeclareMathOperator{\PSL}{PSL}
\DeclareMathOperator{\SL}{SL}

\DeclareMathOperator{\vol}{vol}

\DeclareMathOperator{\tr}{tr}

\newcommand{\csp}{\textbf{c}}
\newcommand{\elp}{\textbf{e}}

\newcommand{\mm}{\mathfrak{a}}
\newcommand{\Z}{\mathcal{Z}}

\begin{document}

\title{An evaluation of the central value of the automorphic scattering determinant}
\author{Joshua S. Friedman\footnote{The views expressed in this article are the author's own and not those of the U.S. Merchant Marine Academy,
the Maritime Administration, the Department of Transportation, or the United States government.}, Jay Jorgenson\footnote{Research supported
by NSF and PSC-CUNY grants.} and Lejla Smajlovi\'{c}}

\maketitle

\begin{abstract}\noindent
Let $M$ be a finite volume, non-compact hyperbolic Riemann surface, possibly with elliptic fixed points, and let $\phi(s)$ denote the automorphic scattering determinant.
From the known functional equation $\phi(s)\phi(1-s)=1$ one concludes that $\phi(1/2)^{2} = 1$.  However, except for the relatively
few instances when $\phi(s)$ is explicitly computable, one does not know $\phi(1/2)$.  In this article we address this
problem and prove the following result.  Let $N$ and $P$ denote the number of zeros and poles, respectively, of $\phi(s)$
in $(1/2,\infty)$, counted with multiplicities.  Let $d(1)$ be the coefficient of the leading term from the Dirichlet series component
of $\phi(s)$.  Then $\phi(1/2)=(-1)^{N+P} \cdot \mathrm{sgn}(d(1))$.
\end{abstract}
%\tableofcontents{}

\section{Introduction}
Various problems and conjectures in number theory can be equated to questions in harmonic analysis.  For instance,
the Riemann hypothesis for the Riemann zeta function is equivalent to the determination of the zeros of the
meromorphic continuation of the (classical) non-parabolic Eisenstein series $E(s,z)$ associated to $\PSL(2,\ZZ)$.  The Lindel\"of hypothesis
for the Riemann zeta function is implied by a sup-norm bound for $E(1/2+ir,z)$ when $r \in \RR$.  As a result, harmonic
analysis associated to quotients of the hyperbolic upper half plane $\HH$ has received considerable interest and study.

Nonetheless, there are many basic questions which remain unsolved.  The far-reaching Phillips-Sarnak philosophy \cite{PS92} asserts that only
in presence of arithmetic or geometric symmetry will there exist square-integrable eigenfunctions of the Laplacian.  Though there
is compelling results supporting this point of view, the main conjectures remain unanswered.  Within this framework, there
are several other questions associated to the harmonic analysis of finite volume quotients of $\HH$ which are seemingly approachable
yet their solutions remain elusive.  The purpose of this paper is to address one such question.

Let $\Gamma$ be a Fuchsian group of the first kind with $\csp$ cusps, and assume $\csp > 0$.  Let $M=\Gamma \backslash \HH$ be the finite volume,
non-compact orbifold quotient space.  Let $\phi(s)$ denote the automorphic scattering determinant, meaning the determinant of the hyperbolic
scattering matrix $\Phi(s)$ which is obtained by computing the various Fourier expansions of the non-holomorphic, parabolic Eisenstein
series associated to $M$.  The function $\phi(s)$ is meromorphic of order at most two.
Furthermore, $\phi(s)$ is holomorphic for $\Re(s) > \frac{1}{2}$, except for a finite number of poles, and it satisfies the functional equation
\begin{equation}\label{functeq phi}
\phi(s)\phi(1-s)=1.
\end{equation}
\noindent
The functional equation implies that $\phi(1/2)^{2} = 1$.  However, except for certain arithmetic groups, one
does not know if $\phi(1/2)$ is one or minus one.  The main result of this article is to evaluate $\phi(1/2)$
in terms of more elementary spectral and group theoretic data associated to $M$.

Let us recall the notation needed to state the main theorem.
For $\Re(s)>1$ the automorphic scattering determinant $\phi(s)$ can be written as an absolutely convergent
generalized Dirichlet series and Gamma functions.  Specifically, for $\Re(s)> 1$ we have that
\begin{equation} \label{phiDirich}
\phi (s)=\pi ^{\frac{\csp}{2}}\left( \frac{\Gamma \left( s-\frac{1}{2}
\right) }{\Gamma \left( s\right) }\right) ^{\csp}\overset{\infty }{\underset
{n=1}{\sum }}\frac{d(n) }{g_{n}^{2s}}
\end{equation}%
where $0< g_{1} < g_{2}< ...$ and $d(n) \in \mathbb{R}$ with $d(1)\neq 0$.

\begin{thm} \label{main thm}
Let $N$ and $P$ denote the number of zeros and poles, respectively, of $\phi(s)$
in $[1/2,\infty)$, counted with multiplicities.  Let $d(1)$ be as in \eqref{phiDirich}. Then
\begin{equation*}\label{phi at 1/2 final formula}
\phi(1/2)=(-1)^{N+P}\mathrm{sgn}(d(1)),
\end{equation*}
where $\mathrm{sgn}$ denotes the sign of a real number.
\end{thm}

We give examples of arithmetic groups where $\phi(1/2) =1$ and where $\phi(1/2) = -1$.  However, it remains to
be seen if $\phi(1/2)$ is constant on any given moduli space, and, if not, then how the
values of $\phi(1/2)$ partition the moduli space.  We leave this problem
to the interested reader.

\section{Preliminary material}

\subsection{Basic notation}
Let $\Gamma\subseteq\mathrm{PSL}(2, \mathbb{R})$ be a Fuchsian group of the first kind acting by fractional
linear transformations on the upper
half-plane $\mathbb{H}:=\{z\in\mathbb{C}\,|\,z=x+iy\,,\,y>0\}$. Let $M$ be the quotient
space $\Gamma\backslash\mathbb{H}$ and $g$ the genus of $M$. Denote by $\csp$ number of inequivalent cusps
of $M$ and by $\{R\}_{\Gamma}$ the set of inequivalent elliptic classes of elements of $\Gamma$. For a
fixed elliptic representative $R$, we denote by $d_R$ the order of element $R$ and by $\elp$ the cardinality of
the finite set $\{R\}_{\Gamma}$  of inequivalent elliptic classes in $\Gamma$.

Recall that the hyperbolic volume $\mathrm{vol}(M)$ of $M$ is given by the Gauss-Bonnet formula
\begin{align*}
\mathrm{vol}(M)=2\pi\bigg(2g-2+\csp + \sum_{\{R\}_{\Gamma}} \left( 1- \frac{1}{d_R}\right)\bigg).
\end{align*}

\vskip .10in
Given a meromorphic function $f(s)$, we define the \emph{null set} $N(f)$ to be $N(f) = \{ s \in \mathbb{C}~|~f(s)=0\}$ counted with multiplicity. Similarly, $P(f)$
denotes the \emph{polar set,} the set of points where $f$ has a pole.

\subsection{The Gamma function} \label{secGamma}
Let $\Gamma(s)$ denote the Gamma function. Its poles are all simple and located at each point of $-\NN,$ where $-\NN = \{ 0,-1,-2,\dots \}$.
For $|\arg{s}| \leq \pi-\delta$ and $\delta > 0$, the asymptotic expansion \cite[p. 20]{AAR99} of $\log{\Gamma(s)}$ is given by
\beq \label{gammaExpan}
\log{\Gamma(s)} = \frac{1}{2}\log{2\pi} + \left(s-\frac{1}{2}\right)\log{s} - s + \sum_{j=1}^{m-1} \frac{B_{2j}}{(2j-1)2j}\frac{1}{s^{2j-1}} + g_{m}(s).
\eeq
Here $B_i$ are the Bernoulli numbers.  Also, for each $m$, $g_{m}(s)$ is a holomorphic function in the right half plane $\Re(s)>0$ such that
$g_{m}^{(j)}(s) = O(s^{-2m+1-j})$ as $\Re(s)\to \infty$ for all integers $j\geq 0$, and where the implied constant depends on $j$ and $m$.

\subsection{The double Gamma function} \label{secBarnes}
The Barnes double Gamma function is an entire order two function defined by
$$
G\left(s+1\right)=\left(2\pi\right)^{s/2}\exp\left[-\frac{1}{2}\left[\left(1+\gamma\right)s^{2}+s\right]\right]\prod_{n=1}^{\infty}\left(1+\dfrac{s}{n}\right)^{n}\exp\left[-s+\frac{s^{2}}{2n}\right],
$$
where $\gamma$ is the Euler constant. Therefore, $G(s+1)$ has a zero of multiplicity $n,$ at each point $-n \in \{-1,-2,\dots \}.$
For $\Re(s)>0$ and as $s \rightarrow \infty,$ the asymptotic expansion of $\log G(s+1)$ is given in \cite{FL01} or \cite[Lemma 5.1]{AD14} by
\begin{equation} \label{asmBarnes}
\log G(s+1) = \frac{s^2}{2}\left( \log{s} - \frac{3}{2}\right) - \frac{\log{s}}{12} - s \, \zeta^{\prime}(0) + \zeta^{\prime}(-1) \>- \\
\sum_{k=1}^{n} \frac{B_{2k+2}}{4\,k\,(k+1)\,s^{2k}} +  h_{n+1}(s).
\end{equation}
Here, $\zeta(s)$ is the Riemann zeta-function and
$$
h_{n+1}(s)= \frac{(-1)^{n+1}}{s^{2n+2}}\int_{0}^{\infty}\frac{t}{\exp(2\pi t) -1} \, \int_{0}^{t^2}\frac{y^{n+1}}{y+s^2} \,dy \,dt.
$$
By a close inspection of the proof of \cite[Lemma 5.1]{AD14}, it follows that $h_{n+1}(s)$ is holomorphic function in the right half plane $\Re(s)>0$
which satisfies the asymptotic relation $h_{n+1}^{(j)}(s) = O(s^{-2n-2-j})$ as $\Re(s)\to \infty$ for all integers $j\geq 0$,
and where the implied constant depends upon $j$ and $n$.

Using the Gamma and Barnes double Gamma function, we can construct a function whose divisor coincides with the trivial zeros and
poles of the Selberg zeta function (see Equation~\eqref{trivZeros}), which we will define later.  Specifically, we are interested
in defining a meromorphic function whose poles are at the negative integers $-n \in -\NN$ with multiplicity
$$m_n= (2n+1)(2g-2+\csp) +2n \elp -  2\sum_{\{R\}_{\Gamma}} \lfloor \frac{n}{d_R}\rfloor,
$$
where $\lfloor x\rfloor$ denotes the integer part of a real number $x$.
To do so, first observe that the function  $$\prod_{m=0}^{d_R - 1}G\left(\frac{s+m}{d_R} + 1\right)$$
has zeros at $-n \in -\NN$ of order $\lfloor \frac{n}{d_R}\rfloor.$ Hence, we define
\begin{equation*} \label{eqBarnesE}
G_E(s) =   \prod_{\{R\}_{\Gamma}} \prod_{m=0}^{d_R - 1}G\left(\frac{s+m}{d_R} + 1\right),
\end{equation*}
and set
\beq \label{def G_1}
G_1(s)= \left(\frac{(2\pi)^{-s} (G(s+1)^2)}{\Gamma(s)}\right)^{2g-2+\csp}\cdot \left( (2\pi)^{-s} (G(s+1))^2 \right)^{\elp} \cdot \left( G_E(s)) \right)^{-2}.
\eeq

It is elementary to show that $G_{1}(s)$ is an entire function of order two with zeros at points
$-n \in -\NN$ and corresponding multiplicities  $m_n.$

\section{Zeta functions}

We are further establishing notation by citing material from the well-known
sources \cite{Hejhal83}, \cite{Iwa02} and \cite{Venkov83}.

\subsection{Automorphic scattering determinant} \label{ascatmatrix}

We will rewrite \eqref{phiDirich} in a slightly different form. Let $c_{1}=-2\log{g_{1}} \neq 0,$  $c_{2}=\log d(1),$  and let $u_{n}=g_{n}/g_{1}>1$.
Then for $\text{Re}(s) > 1$ we can write $\phi( s) =L(s)H(s)$ where
\begin{equation} \label{eqPhiA}
L(s) =\pi^{\frac{\csp}{2}}\left( \frac{\Gamma \left( s-
\frac{1}{2}\right) }{\Gamma \left( s\right) }\right) ^{\csp} e^{c_{1}s+c_{2}}
\end{equation}
and
\begin{equation} \label{eqPhiB}
H(s) =1+\overset{\infty }{\underset{n=2}{\sum }}\frac{
a\left( n\right) }{u_{n}^{2s}},
\end{equation}
where  $a(n) \in \mathbb{R}.$ The series \eqref{eqPhiB} converges absolutely for $\Re(s)>1$.
From the generalized Dirichlet series representation \eqref{eqPhiB} of $H(s)$, it follows that
\beq \label{asmPhi}
\frac{d^{k}}{ds^{k}}\log{H(s)} = O(\beta_k^{-\Re(s)}) \quad \textrm{\rm when} \quad \Re(s) \to +\infty,
\eeq
for some $\beta_k > 1$  where the implied constant depends on $k \in \NN.$

The divisor of $\phi(s)$ consists of the following sets of points:

\begin{enumerate}
\item Finitely many real zeros of the form $1-\sigma_i \in [0,1/2)$ for $i=1\dots T$, each with multiplicity $q(\sigma_i)$;
\item Finitely many real zeros of the form $\rho_i > 1/2$, $i=1\dots N,$ where $N$ is defined to be the sum of the multiplicities;
\item Finitely many real poles of the form $1-\rho_i < 1/2$, where  $i=1\dots N;$
\item Finitely many poles $\sigma_i \in (1/2,1]$, where by the functional equation \eqref{functeq phi} each pole has multiplicity $q(\sigma_i)$ and whose sum is $P$;
\item Poles of the form $1-\rho$ and $1-\overline{\rho}$ with $\Re(\rho) > 1/2$ and $\Im(\rho) > 0;$
\item Zeros of the form $\rho$ and $\overline{\rho}$ with $\Re(\rho) > 1/2$ and $\Im(\rho) > 0$.
\end{enumerate}

%Of course, various sets of zeros and poles correspond in location and multiplicity because of the functional equation $\phi(s)\phi(1-s)=1$.

Let $\lambda_i$ be an eigenvalue for the positive, self-adjoint extension $\Delta$ of the hyperbolic Laplacian.
Denote by $A(\lambda_i)$ the $\Delta-$eigenspace corresponding to the eigenvalue $\lambda_i.$
Set $A_{1}(\lambda_{i})$ to be the subspace of $A(\lambda_i)$ that is spanned by the incomplete theta series.
For each pole $\sigma_{i} \in (1/2,1]$, $i=1,...,T$  the space $A_{1}(\sigma_{i}(1-\sigma_{i}))$   is non-trivial. In fact
from (\cite[Eq. 3.33 on p.299]{Hejhal83}) we have that
\begin{equation*}
q(\sigma_i)=\text{[The multiplicity of the pole of $\phi(s)$ at $s = \sigma_{i}$] } \leq \dim A_{1}(\sigma_{i}(1-\sigma_{i})) \leq \csp.\label{eqBndPole}
\end{equation*}
The eigenvalue $\lambda_i = \sigma_{i}(1-\sigma_{i})$ is called a \emph{residual eigenvalue.}

\subsection{Selberg zeta-function} \label{szeta}
The Selberg zeta function associated to the quotient space $M=\Gamma\backslash\mathbb{H}$ is defined for $\Re(s)>1$ by
the absolutely convergent Euler product
\begin{equation*}
Z(s)=\prod\limits_{\left\{ P_0\right\} \in P(\Gamma
)}\prod_{n=0}^{\infty }\left( 1-N(P_0)^{-(s+n)}\right) \text{,}
\end{equation*}
where $P(\Gamma )$ denotes the set of all primitive hyperbolic conjugacy classes in $\Gamma,$ and $N(P_0)$ denotes the norm of $P_0 \in \Gamma.$
From the product representation given above, we have for $\Re(s)>1$ that

\begin{equation*} \label{log z(s)}
\log{Z(s)} = \sum_{\left\{ P_0\right\} \in P(\Gamma
)} \sum_{n=0}^\infty \left(-\sum_{l=1}^\infty \frac{N(P_0)^{-(s+n)l}}{l}  \right) = -
\sum_{P\in H(\Gamma )}\frac{\Lambda (P)}{N(P)^{s}\log N(P)},
\end{equation*}
where $H(\Gamma )$ denotes the set of all hyperbolic conjugacy classes in $\Gamma,$ and $\Lambda (P)=\frac{\log N(P_{0})}{1-N(P)^{-1}}$, for the
 primitive element $P_{0}$ in the conjugacy class containing $P$.

Let $P_{00}$ be the primitive hyperbolic conjugacy class in all of $P(\Gamma )$ with the smallest norm. Setting $\alpha = N(P_{00})^{\tfrac{1}{2}}$, for $\Re(s) > 2$ and $k \in \NN$ we have the following asymptotic formula
\beq \label{eqSelZetaBound}
\frac{d^{k}}{ds^{k}}\log{Z(s)} = O(\alpha^{-\Re(s)}) \quad \textrm{\rm when} \quad \Re(s) \to +\infty,
\eeq
with an implied constant which depends on $k \in \NN.$

If $\lambda_j$ is an eigenvalue in the discrete spectrum of $\Delta$, let  $m(\lambda_j)$ denote its multiplicty.
We now state the divisor of the $Z(s)$  (see \cite[p. 49]{Venkov90}  \cite[p. 499]{Hejhal83}):

\begin{enumerate}
\item Zeros at the points  $s_j$ on the line $\Re(s)=\tfrac{1}{2}$ symmetric relative to the real axis and in  $(1/2,1]$, where each zero $s_j$ has multiplicity $m(s_j) = m(\lambda_j)$ where $s_j(1-s_j) = \lambda_j$ is an eigenvalue in the discrete spectrum of $\Delta$  \label{szeta1};

\item Zeros at the points $s_{j} \in [0,1/2)$ where $s_j(1-s_j) = \lambda_j \in[0,1/4)$ is an eigenvalue in the discrete spectrum of $\Delta$ and
the multiplicity $\widetilde{m}(s_{j})$ is given by $\widetilde{m}(s_{j}) =  m(\lambda_j) - q(1-s_j) \geq 0$; we denote by $K$ the number of eigenvalues $\lambda_j \in[0,1/4)$ and put $m_j=m(\lambda_j)$, $j=1,...K $.

Note that, in the case when $\lambda_j$ is not the residual eigenvalue, we take $q(1-s_j)= 0$, i.e. $\widetilde{m}(s_{j}) =  m(\lambda_j)$.

\item The point $s=\tfrac{1}{2}$ can be a zero or a pole, and the order of the point as a divisor is
$$\mm = 2d_{1/4}- \tfrac{1}{2}\left( \csp- \tr \Phi (\tfrac{1}{2})\right)$$
where $d_{1/4}$ be the multiplicity  of the possible eigenvalue $\lambda = \tfrac{1}{4}$
of $\Delta$;
\item Poles at $s=-n-\tfrac{1}{2}, $ where $n=0,1,2,\dots,$  each with
multiplicity $\csp$;
\item Finitely many real zeros $1-\rho_i < 1/2$, where  $i=1\dots N;$
\item Zeros at each $s = 1-\rho, 1-\overline{\rho}$ where $\rho$ is a zero of $\phi(s)$ with $\Re(\rho) > \tfrac{1}{2}$ and $\Im(\rho)>0;$ \
 \item Zeros at points $s=-n \in -\NN$,  with multiplicities
    $$
   m_n= \frac{\vol(M)}{2\pi }(2n+1)- \sum_{\{R\}_{\Gamma}}\frac{1}{d_R}\sum_{k=1}^{d_R-1} \frac{\sin\left(\frac{k\pi (2n+1)}{d_R}\right)}{\sin\left(\frac{k\pi}{d_R}\right)}.$$
\end{enumerate}

The last set of zeros are called \emph{trivial} zeros.  It is possible to write the multiplicity $m_n$ in a different way.
By using a double induction in the variables $n$ and $d_R$, one can show that
$$
\sum_{\{R\}_{\Gamma}}\frac{1}{d_R} \left( 2n+1 + \sum_{k=1}^{d_R-1} \frac{\sin\left(\frac{k\pi (2n+1)}{d_R}\right)}{\sin\left(\frac{k\pi}{d_R}\right)} \right) = \sum_{\{R\}_{\Gamma}} \left(2\lfloor \frac{n}{d_R}\rfloor +1\right).
$$
Therefore, by applying the Gauss-Bonnet formula for $\mathrm{vol}(M)$ we immediately get
\beq \label{trivZeros}
m_n= (2n+1)(2g-2+\csp) +2n \elp -  2\sum_{\{R\}_{\Gamma}} \lfloor \frac{n}{d_R}\rfloor .
\eeq

\subsection{Complete zeta functions}  \label{complete zetas}

Define
$$
Z_+(s)=\frac{Z(s)}{G_1(s)(\Gamma(s-1/2))^{\csp}},
$$
where $G_1(s)$ is defined by \eqref{def G_1}.
Note that we have canceled out the trivial zeros and poles of $Z(s)$.
Hence the zero set $N(Z_+)$ of $Z_{+}$consists of the following points:
%%%%%%%%%%%%%%%%%%%%%%%%%%%
\begin{enumerate}
\item At $s = \tfrac{1}{2}$ with multiplicity $\mm$ where
$$
\mm = 2d_{1/4}+\csp-\tfrac{1}{2}\left( \csp-\tr \Phi (\tfrac{1}{2})\right) = 2d_{1/4}+\tfrac{1}{2} \left( \csp +\tr \Phi (\tfrac{1}{2})\right) \geq 0;
$$
\item At the points $s_{j} \in [0,1/2)$ where $s_j(1-s_j) = \lambda_j$ is an eigenvalue in the discrete spectrum of $\Delta$
each with multiplicity $m(\lambda_j) - q(1-s_j) \geq 0;$
\item At the points $s_j$ on the line $\Re(s)=\tfrac{1}{2}$ symmetric relative to the real axis and in  $(1/2,1]$  where
each zero $s_j$ has multiplicity $m(s_j) = m(\lambda_j)$ where $s_j(1-s_j) = \lambda_j$ is an eigenvalue in the discrete spectrum of $\Delta$;
\item At each point $s = 1-\rho,1-\overline{\rho}$ where $\rho$ is a zero of $\phi(s)$ with $\Re(\rho) > \tfrac{1}{2},$ and $\Im(\rho) > 0.$
\end{enumerate}

Sets one, two and three in the above enumeration are finite and are such that all zeros are real.  Within set four, there are a finite number of real zeros.
Hence, in total there are a finite number of real zeros of $Z_{+}$, and the location of the zeros are described in the above four sets.

Define $Z_-(s)= Z_+(s)\phi(s)$.  It follows that $N(Z_-) = 1-N(Z_+)$. In other words,
$s$ is a zero of $Z_+$ if and only if $1-s$ is a zero, necessarily with the same multiplicity, of $Z_{-}$.

\section{Superzeta functions}

\subsection{Regularized products using superzeta functions}
Let $\RR^{-} = (-\infty,0]$ be the non-positive real numbers. Let $\{y_{k}\}_{k\in \mathbb{N}}$ be the sequence of zeros
of an entire function $f$ of order at most two, repeated with their multiplicities. Let
$$X_f = \{z \in \CC~|~ (z-y_{k}) \notin \RR^{-}~\text{for all} ~ y_{k} \}. $$
For $z \in X_f,$ and $s \in \CC$ consider the series
\begin{equation}
\Z_{f}(s,z)=\sum_{k=1}^{\infty }(z-y_{k})^{-s},  \label{Zeta1}
\end{equation}
where the complex exponent is defined using the principal branch of the logarithm with $\arg z\in
\left( -\pi ,\pi \right) $ in the cut plane $\CC \setminus \RR^{-}$.
Since $f$ is of order at most two, the series $\Z_{f}(s,z)$ converges absolutely for $\Re(s) > 2.$
Following \cite{Voros1}, the series $\Z_{f}(s,z)$ is called the \it superzeta function \rm
associated to the zeros of $f $, or the simply the \emph{superzeta} function of $f.$

If $\Z_{f}(s,z)$ has a meromorphic continuation which is regular at $s=0,$ we define the \emph{superzeta regularized product} associated to $f$ as
$$
D_{f}\left( z \right) = \exp\left( {-\frac{d}{ds}\left. \Z_{f}\left( s,z\right) \right|_{s=0}  } \right).
$$

Hadamard's product formula allows us to write
\beq
f(z) = \Delta_{f}(z) = e^{g(z)} z^r \prod_{k=1}^\infty \left( \left(1-\frac{z}{y_k}
\right)\exp\left[ \frac{z}{y_k} + \frac{z^2}{2{y_k}^2}   \right]    \right),
\eeq
where $g(z)$ is a polynomial of degree 2 or less, $r\geq 0$ is the order of eventual zero of $f$ at $z=0,$ and the other zeros $y_k$ are listed with multiplicity.

The following proposition, originally due to Voros ( \cite{Voros1}, \cite{Voros3}, \cite{VorosKnjiga}) is proven in \cite[Prop. 4.1]{FJS16}:

\begin{prop} \label{prop: Voros cont.}
Let $f$ be an entire function of order two, and for $k\in\NN,$ let $y_k$ be the sequence of zeros of $f.$ Let $\Delta_{f}(z)$
denote the Hadamard product representation of $f.$ Assume that for $n>2$ we have the following asymptotic expansion:
\begin{equation} \label{defAE}
\log \Delta_{f}(z)= \widetilde{a}_{2}z^{2}(\log z-\frac{3}{2}%
)+b_{2}z^{2}+\widetilde{a}_{1}z\left( \log z-1\right) +b_{1}z+\widetilde{a}%
_{0}\log z+b_{0}+\sum_{k=1}^{n-1}a_{k}z^{\mu _{k}} + h_n(z),
\end{equation}
where $1>\mu _{1}>...>\mu _{n} \rightarrow -\infty $, and $h_n(z)$ is a sequence of holomorphic functions in the sector $\left\vert \arg z\right\vert <\theta <\pi, \quad (\theta >0)$ such that $h_n^{(j)}(z)=O(|z|^{\mu_n-j})$, as $\left\vert z\right\vert \rightarrow \infty $ in the above sector, for all integers $j \geq 0.$

Then, for all $z\in X_f,$ the superzeta function $\Z_{f}(s,z)$  has a meromorphic continuation to the half-plane $\Re(s)<2$ which is regular at $s=0.$

Furthermore, the superzeta regularized product $D_{f}\left( z\right) $ associated to $f(s)$ is related to $\Delta_{f}(z)$ through the formula
\begin{equation}
\exp\left( {-\frac{d}{ds}\left. \Z_{f}\left( s,z\right) \right|_{s=0}  } \right) = D_{f}(z)=e^{-(b_{2}z^{2}+b_{1}z+b_{0})}\Delta_{f}(z).  \label{D(z)}
\end{equation}
\end{prop}

\subsection{Superzeta functions associated to $Z_+$ and $Z_-$}
Let $X_{\pm} = X_{Z_{\pm}},$ and for $z \in X_{\pm},$ denote by $\mathcal{Z}_{\pm}(s,z) :=\mathcal{Z}_{Z_{\pm}}(s,z)$ the superzeta functions of $Z_{\pm}.$

\begin{thm} \label{thm:cont}
For $z \in X_{\pm},$ the superzeta functions  $\mathcal{Z}_{\pm}(s,z)$ have meromorphic continuations to all of $s\in\CC$, regular at $s=0.$ Furthermore,
for $z\in X_+\cap X_-$
\beq \label{phi expression}
\phi(z)= (\pi)^{\tfrac{\csp}{2}}e^{c_1 z+c_2} \exp\left(-\frac{d}{ds}\left.\left( \Z_-(s,z) - \Z_+(s,z)\right)\right|_{s=0} \right).
\eeq
\end{thm}

\begin{proof}
We claim that $Z_{+}(s)$ and $Z_{-}(s)$ both are entire, order two functions which satisfy the hypothesis of Proposition~\ref{prop: Voros cont.}.
Indeed, the function $G_1(s)$ is a product of rescaled Barnes double Gamma functions, so by using  the asymptotic expansion \eqref{gammaExpan} of the
Gamma function, the expansion \eqref{asmBarnes} of the Barnes double Gamma function, the bound \eqref{eqSelZetaBound} for logarithm of the Selberg zeta function,
and the asymptotic expansion of the logarithm of the automorphic scattering matrix $\phi(s)=L(s)H(s)$, deduced from \eqref{eqPhiA} and \eqref{asmPhi},  we can obtain an asymptotic expansion of the form
\eqref{defAE} for both $Z_+$ and $Z_-$. We refer to the proof of \cite[Thm. 6.2]{FJS16} where similar computations are worked out in complete detail.
Thus, by Proposition~~\ref{prop: Voros cont.} 
both $\mathcal{Z}_{\pm}(s,z)$ have meromorphic continuations to all of $s\in\CC$ which are regular at $s=0.$

Recall that $Z_{-}(s) = \phi(s) Z_{+}(s).$ Hence from the asymptotic properties of $\log{\phi(s)}$, it follows that
$$
\log{Z_{-}(s)} = \log{Z_{+}(s)} + \frac{\csp}{2}\log{\pi} -\frac{\csp}{2}\log{s} + c_1s + c_2 +o(1)
\,\,\,\,\,\text{\rm as $s \rightarrow \infty$}.
$$
Here $c_1$ and $c_2$ are from \eqref{eqPhiA}, and we used the asymptotic expansion of $\log{\Gamma(s-\tfrac{1}{2})}$ which can be obtained
from \eqref{gammaExpan} and Legendre's duplication formula.
Since $\log{Z_{+}(s)}$ has an expansion of the form
$$\log{Z_{+}(s)} = \widetilde{a}_{2}s^{2}(\log s-\frac{3}{2}%
)+b_{2}s^{2}+\widetilde{a}_{1}s\left( \log s-1\right) +b_{1}s+\widetilde{a}%
_{0}\log s+b_{0}+\sum_{k=1}^{n-1}a_{k}s^{\mu _{k}} + h_n(s),$$
we conclude that
$$\log{Z_{-}(s)} = \widetilde{a}_{2}s^{2}(\log s-\frac{3}{2}%
)+b_{2}s^{2}+\widetilde{a}_{1}s\left( \log s-1\right) +(b_{1}+c_1)s+\widetilde{a'}%
_{0}\log s+(b_{0}+\frac{\csp}{2}\log{\pi} + c_2)+\sum_{k=1}^{n-1}a'_{k}s^{\mu _{k}} + g_n(z).$$
Note that of the leading terms only $\widetilde{a}_0,$ $b_1,$ and $b_0$ changed; however,
only the $b-$terms are explicitly present on the right hand side of \eqref{D(z)}.
Hence applying \eqref{D(z)} to both $Z_{-}(s),$ $Z_{+}(s),$ and recalling that $Z_{-}(s) = \phi(s) Z_{+}(s)$
gives us \eqref{phi expression}.

\end{proof}

\section{Proof of the main Theorem}

 As above, let $K$ denote the number of exceptional eigenvalues $\lambda_j$ of the Laplacian, let $s_j\in(1/2,1]$ be such that $s_j(1-s_j)=\lambda_j$ and let $m_j=m(\lambda_j)$ be the multiplicity of the eigenvalue $\lambda_j$, $j=1,...K$.

For any $j\in\{1,...,K\}$, $s_j$ is a zero of $Z_+(s)$ of multiplicity $m_j$ and $(1-s_j)\in[0,1/2)$ is a zero of $Z_+(s)$ of multiplicity $m_j-q(s_j)$ (where we put $q(s_j)=0$ in case when $s_j$ is not a pole of $\phi(s)$). Moreover, recall that $\rho_i>1/2$, $i=1,...N$ are real zeros of $\phi(s)$, counted according to their multiplicities and that the number of poles $\sigma_i \in (1/2,1]$ of $\phi(s)$ is $T\leq K$.

We define
$$
\Z_+^*(s,z):= \Z_+(s,z)- \frac{\mm}{(z-1/2)^s}-\sum_{j=1}^{K}\frac{m_j}{(z-s_j)^s} -\sum_{j=1}^{K}\frac{m_j - q(s_j)}{(z-(1-s_j))^s}- \sum_{i=1}^{N}\frac{1}{(z-(1-\rho_i))^s}
$$
and
$$
\Z_-^*(s,z):=\Z_-(s,z)- \frac{\mm}{(z-1/2)^s}-\sum_{j=1}^{K}\frac{m_j-q(s_j)}{(z-s_j)^s}    -\sum_{j=1}^{K}\frac{m_j}{(z-(1-s_j))^s}
-\sum_{i=1}^{N}\frac{1}{(z-\rho_i)^s}.
$$
From Section~\ref{complete zetas} and proof of Theorem \ref{thm:cont}, we see that, for any fixed  $z\in (X_+\cap X_-)\cup [1/2,1]$ the functions $\Z_\pm^*(s,z)$ are meromorphic in the entire $s-$plane, holomorphic at $s=0$. Hence, by analytic continuation for $z \in \CC,$  equation~\eqref{phi expression} becomes
\begin{multline} \label{phi at 1/2-start}
\phi(1/2) =  (\pi)^{\tfrac{\csp}{2}}e^{(c_1/2+c_2)} \cdot
\exp\left[-\frac{d}{ds} \left(\Z_-^*(s,\tfrac{1}{2}) - \Z_+^*(s,\tfrac{1}{2}) \right. \right. \\ + \left. \left. \left. \sum_{i=1}^{T}q(\sigma_i)\left(\frac{1}{(\sigma_i-\tfrac{1}{2})^s} - \frac{1}{(\tfrac{1}{2}-\sigma_i)^s}   \right) +\sum_{i=1}^{N}\left( \frac{1}{(\tfrac{1}{2}-\rho_i)^s} - \frac{1}{(\rho_i-\tfrac{1}{2})^s}  \right)  \right) \right|_{s=0} \right].
\end{multline}
Note that the support of $q(z)$ is on the set $\{\sigma_1, \sigma_2, \dots, \sigma_T \}.$
Using the principal branch of $\log(z),$ where $-\pi < \arg(z) \leq \pi,$ for $0 \neq z \in \RR,$ it is elementary that
$$\exp\left( \left. \frac{d}{ds} z^{-s} \right|_{s=0}  \right) = \frac{1}{z}. $$
Recalling the definition of $c_1,c_2$ from above, we can simplify \eqref{phi at 1/2-start} and write
\begin{multline} \label{phi1/2b}
\phi(1/2) =  (\pi)^{\tfrac{\csp}{2}}\frac{d(1)}{g_1}
\exp\left(- \frac{d}{ds} \left( \left. \Z_-^*(s,\tfrac{1}{2}) - \Z_+^*(s,\tfrac{1}{2})  \right) \right|_{s=0} \right) \prod_{i=1}^T \left( \frac{\sigma_i-\tfrac{1}{2}}{\tfrac{1}{2}-\sigma_i} \right) ^{q(\sigma_i)} \prod_{i=1}^N \left( \frac{\tfrac{1}{2}-\rho_i}{\rho_i - \tfrac{1}{2}} \right) \\ = (-1)^{P+N}\frac{d(1)}{g_1}(\pi)^{\tfrac{\csp}{2}}
\exp\left(- \frac{d}{ds} \left( \left. \Z_-^*(s,\tfrac{1}{2}) - \Z_+^*(s,\tfrac{1}{2})  \right) \right|_{s=0} \right)
\end{multline}

Next let us look at the difference $\Z_-^*(s,1/2)- \Z_+^*(s,1/2)$. For $\Re(s)>2$,
$$
\Z_-^*(s,1/2)= \sum_{\rho}\frac{1}{(\rho-1/2)^s}+ \sum_{j=1}^{\infty}\left(\frac{1}{(it_j)^s} + \frac{1}{(-it_j)^s} \right)
$$
where the first sum is taken over all zeros $\rho$ of the scattering determinant $\phi(s)$ with $\Re(\rho)>1/2$ and $\Im(\rho)\neq 0$
while the second sum is taken over all real $t_j>0$ such that $\lambda_j=1/4 + t_j^2 >1/4$ are discrete eigenvalues of $\Delta$.
Since the non-real zeros of $Z_-(s)$ come in complex conjugate pairs, for $\Re(s)>2$  we may write
$$
\sum_{\rho: \Im(\rho)>0}\left(\frac{1}{(\rho-1/2)^s} + \frac{1}{(\overline{\rho}-1/2)^s} \right) + \sum_{j=1}^{\infty}\left(\frac{1}{(it_j)^s} + \frac{1}{(-it_j)^s} \right).
$$
Elementary computations show that for real $s>2$, $\Z_-^*(s,1/2)$ is real. By uniqueness of analytic continuation, we deduce that $$\frac{d}{ds}\Z_-^*(s,1/2)|_{s=0}$$ is also real.
Analogously, since the non-real zeros of $Z_+(s)$  also come in complex-conjugate pairs, we deduce that  $$\frac{d}{ds}\Z_+^*(s,1/2)|_{s=0}$$ is also real. Hence
$$
\exp\left(-\frac{d}{ds}\left(\Z_-^*(s,1/2)- \Z_+^*(s,1/2) \right)|_{s=0} \right)= e^{\alpha}>0,
$$
for some $\alpha \in \RR$.  Substituting into \eqref{phi1/2b}, we get that
\begin{equation}\label{phivalue}
\phi(1/2)= (-1)^{N+P} \pi^{\csp/2} \frac{d(1)}{g_1} e^{\alpha}.
\end{equation}
Since we know that $\phi(1/2)^{2} = 1$, it remains to determine the sign of the above expression.
However, since $g_{1}>0$ and $e^{\alpha}$ is positive, we conclude that
$\phi(1/2)=(-1)^{N+P} \cdot \mathrm{sgn}(d(1))$, which completes the proof of the main theorem.

By taking the absolute values of both sides of (\ref{phivalue}), we obtain the following corollary.

\begin{cor} With notation as above, we have that
$$\exp\left(-\frac{d}{ds}\left(\Z_-^*(s,1/2)- \Z_+^*(s,1/2) \right)|_{s=0} \right) = \frac{g_1}{\pi^{\csp/2} |d(1)|}.$$
\end{cor}

\section{Examples}

\begin{example}\rm
In the case when $\Gamma$ is the modular group $\mathrm{PSL}(2,\mathbb{Z})$, the scattering determinant is given by
$$
\phi(s)=\sqrt{\pi}\frac{\Gamma(s-1/2)}{\Gamma(s)} \frac{\zeta(2s-1)}{\zeta(2s)},
$$
where $\zeta(s)$ is the Riemann zeta function, hence
$$
\phi(1/2)= \frac{\sqrt{\pi}}{\Gamma(1/2)} \cdot\zeta(0) \cdot \lim_{s\to 1/2} \frac{\Gamma(s+1/2)}{(s-1/2) \zeta(2s)} = -\frac{1}{2} \cdot \frac{1}{1/2}=-1.
$$

This agrees with Theorem \ref{main thm} since in this case, $P=1$ (there is only one residual eigenvalue of multiplicity one) and there are no real zeros of $\phi(s)$ which are bigger than $1/2$, i.e. $N=0$.
\end{example}

\begin{example}\rm
Let $N$ be a square-free number with $r\geq 1$ distinct prime factors $p_1,...,p_r$.
(Note:  This $N$ is not the same $N$ as in the statement of the main theorem.)
When $\Gamma = \Gamma_0(N)$ is the congruence group, then the scattering matrix $\phi_{N,0}(s)$ is given by
$$
\phi_{N,0}(s)=\left[ \sqrt{\pi}\frac{\Gamma(s-1/2)}{\Gamma(s)} \frac{\zeta(2s-1)}{\zeta(2s)} \right]^{2^r} \prod _{p\mid N} \left( \frac{1-p^{2-2s}}{1-p^{2s}}\right)^{2^{r-1}},
$$
see \cite{Hejhal83}, formula (4.7) on p. 538, hence, obviously $\phi(1/2)=1$.

On the other hand, the first factor in the above equation has a pole at $s=1$ of order $2^r$, while the second factor has a zero at $s=1$ of order $r2^{r-1}$ and there are no other zeros or poles of $\phi(s)$ that belong to the interval $(1/2, +\infty)$. Therefore, for a prime level $N=p$, meaning when $r=1$,
one has $(-1)^{N+P}=-1$ indicating that one has $\mathrm{sgn} (d(1))=-1$. Indeed, if one expands the function
$$
H(s)=\left[\frac{\zeta(2s-1)}{\zeta(2s)} \right]^{2} \left( \frac{1-p^{2-2s}}{1-p^{2s}}\right)
$$
into Dirichlet series, it is easy to see that, actually, $d(1)=-1$.
In case when $N$ possesses two or more distinct prime factors, the number of real zeros and poles of the automorphic
scattering determinant greater than $1/2$ is even, and $d(1)$ is evidentally positive. Therefore, our main result is verified in this case.
\end{example}

\begin{example} \rm Let $\{p_{i}\}$, with $i=1,\ldots,r$, be a set of distinct primes and set
$N=p_1\cdots p_r$.  The subset of $\SL(2,\RR)$, defined by
\begin{align*}
  \Gamma_0(N)^+:=\left\{ e^{-1/2}\begin{pmatrix}a&b\\c&d\end{pmatrix}\in
    \SL(2,\RR): \, ad-bc=e, \,\, a,b,c,d,e\in\Z, \,\, e\mid N,\ e\mid a,
    \ e\mid d,\ N\mid c \right\}
\end{align*}
is an arithmetic subgroup of $\SL(2,\RR)$.  In effect, $\Gamma_{0}(N)^{+}$ is obtained by
adding the Atkin-Lehner involutions to $\Gamma_{0}(N)$.  In \cite{JST14}, it is shown that
the automorphic scattering determinant $\phi_{N}$ associated to $\Gamma_{0}(N)^{+}$ is
  \begin{align*}\label{DefScattViaXi}
    \phi_N(s)=\frac{s}{s-1}\frac{\xi(2s-1)}{\xi(2s)}\cdot \frac{1}{N^s}\cdot\prod_{j=1}^r\frac{p_j^s+p_j}{p_j^s+1},
  \end{align*}
 where $\xi(s):=\frac{1}{2}s(s-1)\pi^{-s/2}\Gamma(s/2)\zeta(s)$ is the completed Riemann zeta function.
We allow $N=1$ in which case $\Gamma_{0}(N)^{+} = \PSL(2,\ZZ)$.
Immediately, one can compute that $\phi_{N}(1/2) = -1$.  For this example, $\phi_{N}(s)$ possesses only one pole at $s=1$ and
no zeros greater than $1/2$, and $\mathrm{sgd}(d(1)) = 1$,  thus verifying our main theorem.

\end{example}

\vspace{5mm}

\noindent
Joshua S. Friedman \\
Department of Mathematics and Science \\
\textsc{United States Merchant Marine Academy} \\
300 Steamboat Road \\
Kings Point, NY 11024 \\
U.S.A. \\
e-mail: FriedmanJ@usmma.edu, joshua@math.sunysb.edu, CrownEagle@gmail.com

\vspace{5mm}
\noindent
Jay Jorgenson \\
Department of Mathematics \\
The City College of New York \\
Convent Avenue at 138th Street \\
New York, NY 10031
U.S.A. \\
e-mail: jjorgenson@mindspring.com

\vspace{5mm}

\noindent
Lejla Smajlovi\'c \\
Department of Mathematics \\
University of Sarajevo\\
Zmaja od Bosne 35, 71 000 Sarajevo\\
Bosnia and Herzegovina\\
e-mail: lejlas@pmf.unsa.ba

\begin{thebibliography}{99}

\bibitem{AD14} V. S. Adamchik, Contributions to the theory of the Barnes function,
\it Int. J. Math. Comput. Sci.  \bf 9(1) \rm (2014),  11-30.

%\bibitem{AJS12}  M. Avdispahi\'c, J. Jorgenson and L. Smajlovi\'c, Asymptotic behavior of the
%Selberg zeta functions for degenerating families of hyperbolic manifolds,
%\it Commun. Math. Phys. \bf 310 \rm (2012) 217--236.

\bibitem{AAR99} G. Andrews, R. Askey, and R. Roy,
Special functions, Vol. \bf 71 \rm Cambridge University Press, 1999.

%\bibitem{AS64} M. Abramowitz and I. Stegun, Handbook of Mathematical
%Functions, With Formulas, Graphs, and Mathematical Tables, NBS Applied
%Mathematics Series 55, National Bureau of Standards, Washington, DC,1964.

%\bibitem{ABMNV87}
%%L. Alvarez-Gaum\'e, J.-B. Bost, G. Moore, P. Nelson, C. Vafa,
%Bosonization on higher genus Riemann surfaces,
%Comm. Math. Phys. 112 (1987), 503 - 552.

%\bibitem{BG94} J. Bolte, C. Grosche, Selberg trace formula for bordered
%Riemann surfaces: hyperbolic, elliptic and parabolic conjugacy classes and
%determinants of Maas-Laplacians, Comm. Math. Phys. 163 (1994), 217-244.

%\bibitem{BS90} J. Bolte, F. Steiner, Determinants of Laplace-like operators
%on Riemann surfaces, Comm. Math. Phys. 130 (1990), 581-597.

%\bibitem{BKK05}
%J. Burgos Gil, J. Kramer, U. K\"uhn,
%Arithmetic characteristic classes of automorphic vector bundles, (English summary)
%Doc. Math. 10 (2005), 619--716.

%\bibitem{BKK07}
%J. Burgos Gil, J. Kramer, U. K\"uhn,
%Cohomological arithmetic Chow rings,
%J. Inst. Math. Jussieu 6 (2007),  1 - 172.

%\bibitem{Che79}
%J. Cheeger, Analytic torsion and the heat equation,
%Ann. of Math. 109 (1979), 259 - 322.

%\bibitem{CM08} G. Cornelissen, M. Marcolli, Zeta functions that hear the
%shape of a Riemann surface, J. Geom. Phys. 58 (2008), 619-632.

%\bibitem{Deninger92} C. Deninger, Local L-factors of motives and regularized
%determinants, Invent. Math. 107 (1992), 135-150.

%\bibitem{D-P86} E. D'Hoker, D. H. Phong, On determinants of Laplacians on
%Riemann surfaces, Comm.Math. Phys. 104 (1986), 537-545.

%\bibitem{Ef88-91} I. Efrat, Determinants of Laplacians on surfaces of finite
%volume, Comm. Math. Phys. 119 (1988), 443-451, Erratum, Comm. Math. Phys.
%138 (1991), 607.

%\bibitem{EKZ12}
%A. Eskin, M. Kontsevich, A. Zorich,
%Sum of Lyapunov exponents of the Hodge bundle with respect to the Teichm\"uller geodesic flow,
%Publ. IHES 120 (2014), 207 - 333.

%\bibitem{Fay81} J. Fay,
%Analytic torsion and Prym differentials,
%Riemann surfaces and related topics: Proceedings of the 1978 Stony Brook Conference (State Univ. New %York, Stony Brook, N.Y., 1978), pp. 107--122,
%Ann. of Math. Stud., 97, Princeton Univ. Press, Princeton, N.J., 1981.

\bibitem{FL01} C. Ferreira and J. Lopez. An asymptotic expansion of the double gamma function,
\it Journal of Approximation Theory \bf 111 \rm (2001), 298-314.

%\bibitem{Freix09} G. Freixas i Montplet,
%An arithmetic Riemann-Roch theorem for pointed stable curves,
%Ann. Sci. \' Ec. Norm. Sup\' er. (4) 42 (2009), no. 2, 335 - 369.

\bibitem{Fisher87} J. Fisher, An approach to the Selberg trace formula via
the Selberg zeta function, Lecture Notes in Mathematics 1253, Spirnger
Verlag, 1987.

%\bibitem{For87} R. Forman, Functional determinants and geometry, Invent.
%Math. 88 (1987), 447-493.

%\bibitem{Frid07} J. S. Friedman, Regularized determinants of the Laplacian
%for cofinite Kleinian groups with finite-dimensional unitary
%representations, Comm. Math. Phys. 275 (2007), no. 3, 659--684.

\bibitem{FJS16} J. S. Friedman, J. Jorgenson, L. Smajlovi\'c, The determinant of the Lax-Phillips scattering operator,
preprint (arxiv.org/abs/1603.07613).

%\bibitem{Gong95} D. Gong, Zeta-determinant and torsion functions on Riemann
%surfaces of finite volume, Manusc. Math. 86 (1995), 435-454.

%\bibitem{GR07} I. S. Gradshteyn abd I. M. Rzyzhik, Table of integrals, series
%and products, seventh ed., Elsevier Academic Press, 2007.

%\bibitem{GZAnnals} L. Guillop\'{e}, M. Zworski, Scattering asymptotics for
%Riemann surfaces. Ann. of Math. (2) 145 (1997), no. 3, 597--660.

%\bibitem{GZGAFA} L. Guillop\'{e}, M. Zworski, The wave trace for Riemann
%surfaces. Geom. Funct. Anal. 9 (1999), no. 6, 1156--1168.

%\bibitem{Hahn09} T. Hahn, An arithmetic Riemann-Roch theorem for metrics with cusps,
%Humboldt Universit\"at Dissertation, 2009.

\bibitem{Hejhal83} D. A. Hejhal, The Selberg trace formula for $PSL(2,%
\mathbb{R})$, vol. II, Lecture Notes in Mathematics 1001, Springer-Verlag,
1983.

%\bibitem{HJL97} J. Huntley, J. Jorgenson, R. Lundelius, On the asymptotic
%behaviour of counting functions associated to degenerating hyperbolic
%Riemann surfaces, J. Funct. Anal. 149 (1997), 58-82.

%\bibitem{Huxley84} M. N. Huxley, Scattering matrices for congruence
%subgroups, In: R. Rankin (ed.) Modular forms, pp. 141-156, Elis Horwood,
%Chichester, 1984.

%\bibitem{Illies01} G. Illies, Regularized products and determinants, Comm. Math. Phys. 220 (2001), 69-94.

\bibitem{Iwa02} H. Iwaniec, Spectral methods of automorphic forms. Second
edition. Graduate Studies in Mathematics, 53. American Mathematical Society,
Providence, RI; Revista Matem\'{a}tica Iberoamericana, Madrid, 2002.

%\bibitem{Jorg91} J. Jorgenson,
%Analytic torsion for line bundles on Riemann surfaces,
%Duke Math. J. 62 (1991), 527 - 549.

%\bibitem{JL93a} J. Jorgenson, S. Lang,On Cramer's theorem for
%general Euler products with functional equation, Math. Ann. 297 (1993),
%383-416.

%\bibitem{JL93b} J. Jorgenson, S. Lang, Basic analysis of regularized
%series and products, Lecture Notes in Mathematics 1564, 1-122, Spirnger
%Verlag, 1993.

%\bibitem{JL99} J. Jorgenson, S. Lang, Hilbert-Asai Eisenstein series,
%regularized products, and heat kernels. Nagoya Math. J. 153 (1999), 155--188.

%\bibitem{JLund95} J. Jorgenson, R. Lundelius,
%Convergence theorems for relative spectral functions on hyperbolic Riemann surfaces of finite %volume,
%Duke Math. J. 80 (1995), 785-819.

%\bibitem{JLund96} J. Jorgenson, R. Lundelius, Continuity of relative hyperbolic
%spectral theory through metric degeneration, Duke
%Math. J. \bf 84 \rm (1996) 47-81.

%\bibitem{JLund97} J. Jorgenson, R. Lundelius, Convergence of the normalized
%spectral counting function on degenerating hyperbolic Riemann surfaces of
%finite volume, J. Funct. Anal. 149 (1997), no. 1, 25--57.

\bibitem{JST14} J. Jorgenson, L. Smajlovi\'c, H. Then, On the distribution of eigenvalues of
Maass forms on certain moonshine groups,
\it Mathematics of Computation \bf 83 \rm (2014), 3039--3070.

%\bibitem{KimWak04} K. Kimoto, M. Wakayama, Remarks on zeta regularized
%products, Int. Math. Res. Not. 17 (2004), 855-875.

%\bibitem{KK09}
%A. Kokotov, D. Korotkin, Tau-functions on spaces of abelian differentials and higher genus
%generalizations of Ray-Singer formula, J. Differential Geom. 82 (2009), 35-–100.

%\bibitem{KurWak04} N. Kurokawa, M. Wakayama, Zeta regularizations,
%Acta App. Math. 81 (2004), 147-166.

%\bibitem{KurWakYam} N. Kurokawa, M. Wakayama, Y. Yamasaki, Milnor-Selberg
%zeta functions and zeta regularizations, J. Geom. Phys 64 (2013) no. 1, 120 - 145.

%\bibitem{Lax-Phill76} P. Lax, R. Phillips, Scattering theory for automorphic
%functions, Annals of Mathematics Studies No. 87., Princeton Univ. Press,
%Princeton, N.J., 1976.

%\bibitem{LaxPhill80} P. Lax, R. Phillips, Scattering theory for automorphic
%functions, Bull. Amer. Math. Soc. (N.S.) 2 (1980), no. 2, 261--295.

%\bibitem{Lund90} R. E. Lundelius, Asymptotics of the determinant of the
%Laplacian on hyperbolic surfaces of finite volume, Ph. D. Thesis, Stanford
%University, Stanford 1990.

%\bibitem{MTa06}
%A. McIntyre, L. Takhtajan,
%Holomorphic factorization of determinants of Laplacians on Riemann surfaces and a higher genus %generalization
%of Kronecker's first limit formula,
%Geom. Funct. Anal. 16 (2006), 1291 - 1323.

%\bibitem{MTe08}
%A. McIntyre, L.-P. Teo,
%Holomorphic factorization of determinants of Laplacians using quasi-Fuchsian uniformization,
%Lett. Math. Phys. 83 (2008), 41 - 58.

%\bibitem{Mellin} Hj. Mellin, \textit{\"{U}ber die Nullstellen der
%Zetafunktion}, Ann. Acad. Sci. Fenn. A10 no. 11 (1917).

%\bibitem{Mil66} J. Milnor,
%Whitehead torsion, Bull. Amer. Math. Soc. 72 1966 358 - 426.

%\bibitem{Milnor} J. Milnor, On polylogarithms, Hurwitz zeta functions and \
%the Kubvert identities, Enseignment Math\'{e}matique, 29 (1983), 281-322.

%\bibitem{MP49} S. Minakshisundaram, A. Pleijel,
%Some properties of the eigenfunctions of the Laplace-operator on Riemannian manifolds,
%Canadian J. Math. 1, (1949), 242 - 256.

%\bibitem{Mul78}  W. M\"{u}ller,
%Analytic torsion and R-torsion of Riemannian manifolds,
%Adv. in Math. 28 (1978), 233 - 305.

%\bibitem{MulMul06} J. M\"{u}ller, W. M\"{u}ller, Regularized determinants of
%Laplace-type operators, analytic surgery, and relative determinants, Duke
%Math. J. 133 (2006), 259 - 312.

%\bibitem{Mull92} W. M\"{u}ller, Spectral geometry and scattering theory for
%certain complete surfaces of finite volume, Invent. Math. 109 (1992), no. 2,
%265 - 305.

%\bibitem{Mull98} W. M\"{u}ller, Relative zeta functions, relative
%determinants and scattering theory, Comm. Math. Phys. 192 (1998), no. 2,
%309--347.

%\bibitem{OPS89}
%B. Osgood, R. Phillips, P. Sarnak,
%Moduli space, heights and isospectral sets of plane domains,
%Ann. of Math.  129 (1989), 293 - 362.

%\bibitem{Pet94} Y. N. Petridis, On the singular set, the resolvent and
%Fermi's Golden Rule for finite volume hyperbolic surfaces, Manusc. Math. 82
%(1994), 331-347.

%\bibitem{P-S85} R. Phillips, P. Sarnak, On cusp forms for co-finite
%subgroups of $PSL(2,\mathbb{R})$.

\bibitem{PS92} R. Phillips, P. Sarnak, Pertubation theory for the laplacian
on automorphic functions, \it J. Am. Math. Soc. \bf 5 \rm (1992), no. 1, 1-32.

%\bibitem{Pol-Roch-97} M. Pollicott, A. C. Rocha, A remarkable formula for
%the determinant of the Laplacian, Invent. Math. 130 (1997), 399-414.

%\bibitem{Qu85} D. Quillen,
%Determinants of Cauchy-Riemann operators on Riemann surfaces,
%Functional Anal. Appl. 19 (1985), no. 1, 31 - 34.

%\bibitem{RaySing71} D. Ray, I. Singer,
%R-torsion and the Laplacian on Riemannian manifolds,
%Advances in Math. 7 (1971), 145 - 210.

%\bibitem{RaySing73} D. Ray, I. Singer, Analytic torsion for analytic
%manifolds, Ann. Math. 98 (1973), 154-177.

%\bibitem{Sarnak} P. Sarnak, Determinants of Laplacians, Comm. Math. Phys.
%110 (1987), 113 - 120.

%\bibitem{Sel56} A. Selberg, Harmonic analysis and discontinuous groups in
%weakly symmetric Riemannian spaces with applications to Dirichlet series. J.
%Indian Math. Soc. (N.S.) 20 (1956), 47--8.

%\bibitem{Sel 90} A. Selberg, Remarks on the distribution of poles of
%Eisenstein series. Festschrift in honor of I. I. Piatetski-Shapiro on the
%occasion of his sixtieth birthday, Part II (Ramat Aviv, 1989), 251--278,
%Israel Math. Conf. Proc., 3, Weizmann, Jerusalem, 1990.

%\bibitem{Uetake08} Y. Uetake, Lax-Phillips scattering for automorphic
%functions based on the Eisenstein transform, Integr. Equ.Oper. Theory 60
%(2008), 271-288.

%\bibitem{Vardi88} I Vardi, Determinants of Laplacians and multiple gamma functions,
%SIAM J. Math. Anal. 19 (1988), 493-507.

\bibitem{Venkov83} A. Venkov, Spectral theory of automorphic functions, Vol. 153,
American Mathematical Soc., (1983).

\bibitem{Venkov90} A. Venkov, Spectral theory of automorphic functions: and its applications Vol. 51, Klower Academic Publishers, (1990).

\bibitem{Voros1} A. Voros, Spectral functions, special functions and
the Selbeg zeta functions, \it Commun. Math. Phys. \bf 110 \rm (1987), 439-465.

\bibitem{Voros2} A. Voros, Zeta functions for the Riemann zeros,
\it Ann. Inst. Fourier (Grenoble) \bf 53 \rm (2003), 665 - 699.

\bibitem{Voros3} A. Voros, More Zeta functions for the Riemann zeros
I, In: Frontiers in Number Theory, Physics and Geometry, Springer Berlin
(2006), 349-363.

\bibitem{VorosKnjiga} A. Voros, Zeta functions over Zeros of Zeta
Functions, Lecture Notes of the Unione Matematica Italiana, Springer Verlag,
2010.

%\bibitem{Wol-I-92} S. A. Wolpert, Spectral limits for hyperbolic surfaces,
%I, Invent. Math. 108 (1992), 67-89.

%\bibitem{Wol-II-92} S. A. Wolpert, Spectral limits for hyperbolic surfaces,
%II, Invent. Math. 108 (1992), 91-129.

%\bibitem{Yo99}
%K. Yoshikawa,
%Discriminant of theta divisors and Quillen metrics,
%J. Differential Geom. 52 (1999), 73 - 115.

%\bibitem{Yo04}
%K. Yoshikawa,
%K3 surfaces with involution, equivariant analytic torsion, and automorphic forms on the moduli space,
%Invent. Math. 156 (2004), no. 1, 53 - 117.

%\bibitem{Yo13}
%K. Yoshikawa,
%A trinity of the Borcherds Φ-function. Symmetries, integrable systems and representations, 575 - %597,
%Springer Proc. Math. Stat., 40, Springer, Heidelberg, 2013.


\end{thebibliography}
\end{document}